\documentclass[11pt]{article}
\usepackage{amssymb,amsmath}
 \usepackage{fullpage}
\usepackage{tikz}
\usepackage[left,mathlines,modulo]{lineno}

\newtheorem{theorem}{\bf Theorem}[section]
\newtheorem{corollary}[theorem]{\bf Corollary}
\newtheorem{lemma}[theorem]{\bf Lemma}
\newtheorem{proposition}[theorem]{\bf Proposition}

\newtheorem{definition}[theorem]{\bf Definition}

\newcommand{\qed}{\hfill $\square$ \bigskip}
\newcommand{\Fib}{{\cal F}}

\newcommand{\ZZ}{\mathbb{Z}}

\newcommand{\A}{{\mathcal A}}
\newcommand{\supp}{{\rm supp}}

\newcommand{\st}{ ~|~ }

\numberwithin{equation}{section}

\begin{document}
\modulolinenumbers[5]
\title{$p$-th order generalized Fibonacci cubes 
 and maximal cubes in Fibonacci $p$-cubes}
\author{
Michel Mollard\footnote{Institut Fourier, CNRS, Universit\'e Grenoble Alpes, France email: michel.mollard@univ-grenoble-alpes.fr}
}
\date{\today}
\maketitle

\begin{abstract}
The Fibonacci cube $\Gamma_n$ is the subgraph of the hypercube $Q_n$ induced by vertices with no consecutive $1$s. We study a one parameter generalization, $p$-th order Fibonacci cubes $\Gamma^{(p)}_n$, which are subgraphs of  $Q_n$ induced by strings without $p$ consecutive 1s. We show the link  between vertices of  $\Gamma^{(p)}_n$  
and compositions of integers with parts in $\{1,2,\dots,p\}$. Among other eumerative properties, we study the order, size and cube polynomial of $\Gamma^{(p)}_n$ as well as their generating functions. Many of the given expressions are similar to those for Fibonacci cubes, where the p-nomial coefficients play the role of binomial coefficients. We also show that maximal induced hypercubes in Fibonacci $p$-cubes $\Gamma^{p}_n$, another generalization of Fibonacci cubes, are connected to vertices of $p+1$-th order Fibonacci cubes. We use this link to determine the maximal cube polynomial of Fibonacci $p$-cubes.
\end{abstract}


\noindent
{\bf AMS Subj. Class. }:  05C12, 05C30, 05C31, 05C60, 05C90

\section{Introduction }

The {\em Fibonacci cube} of dimension $n$, denoted as $\Gamma_n$, is the subgraph of the hypercube $Q_n$ induced by vertices with no consecutive 1s. This graph was introduced in \cite{H-1993a} as an interconnection network.

$\Gamma_n$ is an isometric subgraph and is inspired by the Fibonacci numbers. It has attractive recurrent structures such as
its decomposition into two subgraphs which are also Fibonacci cubes themselves.  Structural properties of these graphs were more extensively 
studied afterwards; see for example the survey~\cite{K-2013a} and the recent book~\cite{EKM-2023}. \\

Not only the investigation of the 
properties of Fibonacci cubes  attracted many researchers, but it
has also led to the development of a variety of interesting 
generalizations and variations covered by a whole chapter in the book
\cite{EKM-2023}. Among these families of graphs are the cyclic version the Lucas cubes,
generalized Fibonacci cubes~\cite{IKR-2012a} and Fibonacci $p$-cubes~\cite{WY-2022}.

Suppose $f$ is an arbitrary binary string. The generalized Fibonacci cube $Q_n(f)$ is the subgraph of $Q_n$ induced by strings of length $n$ that do not admit $f$ as substring.
 
In particular, for an integer $p$, the graph $Q_n(1^p)$ has been introduced already as an interconnection network by Hsu and Chung ~\cite{HC-1993}. Following Salvi~\cite{S-1996} this graph is called  $p$-th order Fibonacci cube (of dimension $n$) in this paper. We will also use for it the notation $\Gamma^{(p)}_n$ proposed by Belbachir and Ould-Mohamed~\cite{BO-2020}.
 
Rooting properties of $p$-th order Fibonacci cubes have been studied in~\cite{WG-2003}.
They are  are mostly Hamiltonian~\cite{LWC-1994} and even bipancyclic~\cite{S-1996}.
$\Gamma^{(p)}_n$ is a  median graph only in  the cases $p=2$ of Fibonacci cubes and  for  arbitrary $p$ and  dimension $n\leq p$~\cite{OZ-2013}. $p$-th order Fibonacci cubes are daisy cubes~\cite{KM-2019a}.

Some enumerative properties like order, size and cube polynomial have been determined for the case $p=3$ in~\cite{BO-2020} where they are called Tribonacci cubes. A linear time algorithm for the recognition of Tribonacci cubes is given in~\cite{RV-2016}.

In this article, we propose to study enumerative properties for the general case of $p$-th order Fibonacci cubes using an approach different from that used for Tribonacci cubes in~\cite{BO-2020}.

This paper is organized as follows. 
After  preliminaries in Section~\ref{sec:basic}, we will recall some results about compositions and $p$-nomial coefficients. We will study the weight distribution of vertices $\Gamma^{(p)}_n$ in Section~\ref{sec:enumprop} and edges in Section~\ref{sec:edge}. The cube polynomial of $\Gamma_n^{(p)}$ is determined in Section~\ref{sec:distancecube}.
   
Recently~\cite{WY-2022} Wei and  Yang introduced  Fibonacci $p$-cubes $\Gamma_n^p$  which are subgraphs of hypercubes induced by strings where there is at least $p$ consecutive $0$s between  two $1$s. In  Section~\ref{sec:maximalcube} we will study maximal hypercubes in   $\Gamma_n^p$ exhibiting a bijection  between maximal hypercubes in Fibonacci $p$-cubes $\Gamma^{p}_n$ and vertices of weight $n-(p+1)k+p$ in the $p+1$-th order Fibonacci cube $\Gamma^{(p+1)}_{n-pk+p}$.

\section{Preliminaries}
\label{sec:basic}
We will next give  some concepts and notations needed in this paper.

We denote by $[\![a,n]\!]$ the set of integers $i$ such that $a\leq i \leq n$.

Let $(F_n)_{n\geq0}$ be the \emph{Fibonacci numbers}:
$F_0 = 0$, $F_1=1$, $F_{n} = F_{n-1} + F_{n-2}$ for $n\geq 2$.

Let $B=\{0,1\}$. If will be convenient to identify elements $u = (u_1,\ldots, u_n)\in B^n$ and strings of length $n$ over $B$. We thus  briefly write $u$ as $u_1\ldots u_n$ and call $u_i$ the $i$th coordinate of $u$. For $j\in[\![1,n]\!]$ the string  $u+\delta_j$ will be the string $v$ such that $v_j\neq u_j$ and $v_i=u_i$ for all $i\neq j$.

We will use the power notation for the concatenation of bits, for instance $0^n = 0\ldots 0\in B^n$. Let $s$ be a binary string and ${\cal A}$ a set of binary strings we will use the notations  $s {\cal A}$ and ${\cal A} s$ for the sets $s {\cal A}=\{su \st u\in {\cal A}\}$ and ${\cal A} s=\{us \st u\in {\cal A}\}$.

The vertex set of $Q_n$, the \emph{hypercube of dimension $n$},  is the set $B^n$, two vertices being adjacent iff they differ in precisely one coordinate. We will say that an edge $uv$ of $Q_n$ uses the direction $i$ if $u$ and $v$ differ in the coordinate $i$, thus if $v=u+\delta_i$. 

The \emph{distance} between two vertices $u$ and $v$ of  a graph $G$  is the 
number of edges on a shortest $u,v$-path. It is immediate that the distance between two vertices of $Q_n$ is the number of coordinates the strings differ, sometime called Hamming distance.

A {\em Fibonacci string} is a binary string without consecutive 1s. We will call ${\cal F}_n$ the set of Fibonacci strings of length $n$.  

The {\em Fibonacci cube} $\Gamma_n$ ($n\geq 1$) is the subgraph of $Q_n$ induced by $\Fib_{n}$ the set of Fibonacci strings of length $n$. Because of the empty string $\lambda$, $\Gamma_0 = K_1$. 

It is well known that for any integer $n$, $|\Fib_{n}|=|V(\Gamma_{n})|=F_{n+2}$.

Many generalizations of Fibonacci cubes have been proposed. Among them daisy cubes which we will recall at the end of this section and Fibonacci $(p,r)$-cubes~\cite{EA-1997}. Wei and Yang studied specifically Fibonacci $p$-cubes which are Fibonacci $(p,r)$-cubes with $r=1$~\cite{WY-2022}. Other results about Fibonacci $p$-cubes can be found in~\cite{M-2025b}.

For an integer $p\geq1$, A {\em Fibonacci $p$-string} of length $n$ is a binary string where  consecutive 1s are separated by at least $p$ 0s.
 Let ${\cal F}^p_n$ be the set of Fibonacci $p$-strings of length $n$. Then the {\em Fibonacci $p$-cube}, $\Gamma^p_n$ is the subgraph of $Q_n$ induced by $\Fib^p_{n}$. Again because of the empty string  $\Gamma^p_0 = K_1$. Note that the Fibonacci $1$-cube $\Gamma^1_n$ is the classical Fibonacci cubes $\Gamma_n$.


Let $(F^p_{n})_{n\geq 0}$ be the \emph{Fibonacci $p$-numbers} defined by the recursion
\begin{equation*}
F^p_0 = 0, F^p_i=1\text{ for }i\in[\![1,p]\!]\text{, and }F^p_{n} = F^p_{n-1} + F^p_{n-p-1}\text{ for }n\geq p+1. 
\end{equation*}

From this definition, by the usual method, the generating function of the sequence $(F^p_{n})_{n\geq 0}$ is 
\begin{equation*}
\sum_{n\geq 0}F^p_{n}t^n=\frac{t}{1-t-t^{p+1}}.
\end{equation*}

It is immediate that $F^p_{p+1}=1$ and more generally $F^p_{n}=n-p$ for $n\in [\![p+1,2p+2]\!]$. Note that the $(F^1_{n})_{n\geq 0}$ are the Fibonacci numbers.

As noticed in \cite[Theorem~4.1]{WY-2022} the order of $\Gamma^p_n$ is $|{\cal F}^p_n|=F^p_{n+p+1}$.

Indeed  a string in ${\cal F}^p_n$, $n\geq p+1$, can be be uniquely decomposed as the concatenation of $10^p$ with a string of ${\cal F}^p_{n-p-1}$ or as the concatenation of $0$ with a string of ${\cal F}^p_{n-1}$. Furthermore $|{\cal F}^p_n|=n+1=F^p_{n+p+1}$ for $n\leq p$, thus by induction  $|{\cal F}^p_n|=F^p_{n+p+1}$ for any $n$.

An other classic generalization of Fibonacci numbers is to consider sequences in which each sequence element is the sum of the previous $p$ elements. For an integer $p\ge2$ the \emph{$p$-th order generalized Fibonacci numbers} $F_n^{(p)}$ are then  defined by 
\begin{eqnarray*}
F_n^{(p)}=0\text{ for }n\in[\![0,p-2]\!] \text{ and } F_{p-1}^{(p)}=1\nonumber \\ F_n^{(p)}=F_{n-1}^{(p)}+F_{n-2}^{(p)}+\dots+F_{n-p}^{(p)}\text{ for }n\geq p.
\end{eqnarray*}
(see OEIS sequences A000073, A000078, A001591 and array A0922921).  The $(F^{(2)}_{n})_{n\geq 0}$ are the Fibonacci numbers and  obviously  $F_n^{(p)}=2^{n-p}$ for $n\in[\![p,2p-1]\!]$.

It is immediate that the generating function of the sequence $(F^{(p)}_{n})_{n\geq 0}$ is 
\begin{equation}\label{eq:gen(p)}
\sum_{n\geq 0}F^{(p)}_{n}t^n=\frac{t^{p-1}}{1-t-t^2-\dots-t^{p}}.
\end{equation}

For an integer $p\geq2$, a {\em $p$-th order Fibonacci string} is a binary string that does not contain $1^p$ as a substring. In other words $p$-th order Fibonacci strings  are strings with at most $p-1$ consecutive $1$s.

 Let ${\cal F}^{(p)}_n$ be the set of $p$-th order Fibonacci strings of length $n$. Note that ${\cal F}^{(p)}_n=B^n$ for $n\leq p-1$. Then the {\em generalized Fibonacci cube} $\Gamma^{(p)}_n=Q_n(1^p)$  is the subgraph of $Q_n$ induced by ${\cal F}^{(p)}_n$. 
Again because of the empty string  $\Gamma^{(p)}_0 = K_1$ and more generally $\Gamma^{(p)}_n = Q_n$ for $n\leq p-1$. Note that  the $\Gamma^{(2)}_n$ are the classical Fibonacci cubes $\Gamma_n$.

The {\em Hamming weight} $w(b)$ of a binary string $b$ is  the number of occurrences of $1$ in $b$. 
It is immediate that the number of strings in $B^n$ of Hamming weight $w$ is $\binom{n}{w}$.

The following result is well-known (see~\cite{HIK-2011} for example).

\begin{proposition}\label{pro:bt}
In every induced subgraph $H$ of $Q_n$ isomorphic to $Q_k$ there exists a unique vertex of minimal Hamming weight, \emph{the bottom vertex} $b(H)$. There exists also a unique vertex of maximal Hamming weight, the \emph{top vertex} $t(H)$. 
Furthermore $b(H)$ and $t(H)$ are at distance $k$ and characterize $H$ among the subgraphs of $Q_n$ isomorphic to $Q_k$. 
\end{proposition}

If $H$ is an induced subgraph of $\Gamma_n$, it is also an induced subgraph of $Q_n$. Thus Proposition~\ref{pro:bt} is still true for induced subgraphs of Fibonacci cubes.
The {\em support} of an induced hypercube $H$, $\supp(H)$, is the set of $k$ coordinates that vary in $H$. Therefore, 
$$\supp(H) = \{i\in [1,n]:\ t_i=1, b_i=0\}\,.$$ 
$H$ is thus characterized by the couple $(t,b)$ consisting of the top vertex and the bottom vertex of $H$.

If $u$ and $v$ are vertices of a graph $G$, the \emph{interval} $I_G(u,v)$ between $u$ and $v$ (in $G$) is the set of vertices lying on shortest $u,v$-path, that is, $I_G(u,v) = \{w | d(u,v) = d(u,w) + d(w,v)\}$. We will also write $I(u,v)$ when $G$ will be clear from the context. 
A subgraph $G$ of a graph $H$ is an \emph{isometric subgraph} 
if the distance between any vertices of $G$ equals the distance 
between the same vertices in $H$. 
Isometric subgraphs of hypercubes are called \emph{partial cubes}. 
The {\em dimension} of a partial cube $G$ is the smallest integer
$d$ such that $G$ is an isometric subgraph of $Q_d$. 
Many important classes of graphs are partial cubes, 
in particular trees, median graphs, benzenoid graphs, phenylenes, 
grid graphs and bipartite torus graphs. In addition, Fibonacci  
and Lucas cubes are partial cubes as well, see \cite{K-2005}.

For a graph $G$, let $c_k(G)$ $(k\ge 0)$ be the number of induced 
subgraphs of $G$ isomorphic to  $Q_k$. The {\em cube polynomial}, $C_G(x)$,
of $G$, is the corresponding enumerator polynomial, that is

\begin{equation}\label{eqn:defCG}
C_G(x) = \sum_{k\geq 0} c_k(G) x^k\,.
\end{equation}
This polynomial was introduced in~\cite{BKS-2003}, determined for Fibonacci and Lucas cubes in~\cite{KM-2012a} and afterwards for several of their variations~(see \cite[Chapter~9]{EKM-2023} for example). The expression of the cube polynomial of $\Gamma^p_n$ was conjectured in~\cite{WY-2022}, conjecture  proved in~\cite{M-2025b}.

Assume $0^n$ belongs to $G$ subgraph of $Q_n$. A bivariate refinement of $C_{G}(x)$ is the {\em distance cube polynomial} (with respect to $0^n$)~\cite{KM-2019a}, first introduced as {\em$q$-cube polynomial}{~\cite{SE-2017a,SE-2018a} in the case of graphs where it can be seen as a $q$-analogue of the cube polynomial. This polynomial keeps track of the distance of the hypercubes to $0^n$.  We define the polynomial the following way

\begin{equation}\label{eqn:defCGq}
D_G(x,q) = \sum_{k\geq 0} c_{k,d}(G) x^kq^d 
\end{equation} where $c_{k,d}(G)$ is  the number of induced 
subgraphs of $G$ isomorphic to $Q_k$ with bottom vertex at distance $d$ of $0^n$.

A {\em maximal hypercube}  of a graph $G$ is an induced subgraph $H$ isomorphic to a hypercube such that $H$ is not contained in a larger induced hypercube of $G$. For a given interconnection topology it is important to characterize maximal hypercubes, for example from the point of view of embeddings.

Let $h_k(G)$ be the number of maximal hypercubes of dimension $k$ of $G$ and $H_{G}(x)$ the corresponding 
enumerator polynomial, that is, 

$$H_{G}(x) = \sum_{k\geq 0} h_k(G) x^k\,.$$

This polynomial was determined for Fibonacci and Lucas cubes~\cite{M-2012a}, alternate Lucas-cubes~\cite{ESS-2021e} and Pell graphs~\cite{EKM-2023}.

If $G=(V(G),E(G))$ is a graph and $X\subseteq V(G)$, then $\langle X\rangle$ denotes the subgraph of $G$ induced by $X$.  
Let $\le$ be a partial order on $B^n$ defined with $u_1\ldots u_n \le v_1\ldots v_n$ if $u_i\le v_i$ holds for $i\in [1,n]$. For $X \subseteq B^n$ we define the {\em daisy cube generated by $X$} as the subgraph of $Q_n$  
$$\left\langle \{u\in B^n| u\le x\ {\rm for\ some}\ x\in X \} \right\rangle\,.$$
 Finally we will say that a graph $G$ is {\em a daisy cube} if there exist an isometrical embedding  of $G$ in some hypercube $Q_n$ and a subset $X$ of $B^n$ such that $G$ is the daisy cube generated by $X$. Such an embedding will be called a {\em proper embedding}.

By construction daisy cubes are partial cubes. It is immediate that the class of daisy cubes is closed under the Cartesian product.

Fibonacci cubes, Lucas cubes, Alternate Lucas-cube~\cite{ESS-2021e} and  Pell graph~\cite{M-2019}\cite[Theorem 9.68 for a proof]{EKM-2023} are examples of daisy cubes.

\section{$p$-nomial coefficients and compositions }
\label{sec:pnomial}
The concepts presented in this section will be used throughout this document. They are taken from the literature although their first occurrence is difficult to find. We will recall these definitions and results to fix the notations. Moreover, since proofs of the properties are easy, we will give them so that this article is standalone.

A classical way to generalize binomial coefficients is to consider polynomial coefficients~\cite{A-1974}, more precisely  $p$-nomial coefficients for a fixed integer $p\geq2$. The study of them begin with de~Moivre~\cite{D-1756} and Euler~\cite{E-1801}. The notations used by modern authors are not always the same; for our purposes it will be convenient to use the following one which is the most common. Note that Andrews\cite{A-1990} uses this notation for the centered $p$-nomial coefficients. 
\begin{definition}
 For any integers $p$, $a$ and $b$ with $b\geq 0$ and $p\geq 2$, the \emph{$p$-nomial coefficient $\binom{b}{a}_{p-1}$} is the coefficient of $x^a$ in $(1+x+x^2+\dots +x^{p-1})^b$.
\end{definition}
The 2-nomial coefficients are the binomial coefficients  $\binom{b}{a}_1=\binom{b}{a}$. It is also immediate that $\binom{b}{a}_{p-1}= 0 $ for $a>b(p-1)$ or $a< 0$, $\binom{b}{0}_{p-1}=\binom{b}{b(p-1)}_{p-1}=1$ and $\binom{b}{a}_{p-1}=\binom{b}{b(p-1)-a}_{p-1}$. Likewise many properties of binomial coefficients can be generalized to $p$-nomial coefficients, see~\cite{F-2015}. For example from the definition they satisfied for $b\geq1$ and any $a\in\ZZ$ the recurrence relation 
\begin{equation}\label{eqn:pascal}
\binom{b}{a}_{p-1}=\binom{b-1}{a}_{p-1}+\binom{b-1}{a-1}_{p-1} +\dots +\binom{b-1}{a-p+1}_{p-1}.
\end{equation}
 They thus can be obtained from $\binom{0}{0}_{p-1}=1$ and $\binom{0}{a}_{p-1}=0$, $a\neq0$, by a generalized Pascal triangle~\cite{E-1801, H-1969} (see Figure~\ref{fig:pascal}). The $p$-nomial coefficients can also be expressed using binomial coefficients, for example by the following formula derived from the negative-binomial theorem.
\begin{equation*}
\binom{b}{a}_{p-1}=\sum_{i=0}^{\left\lfloor \frac{a}{p}\right\rfloor}{(-1)^i}\binom{b}{i}\binom{b+a-1-i\,p}{b-1}\,.
\end{equation*} 

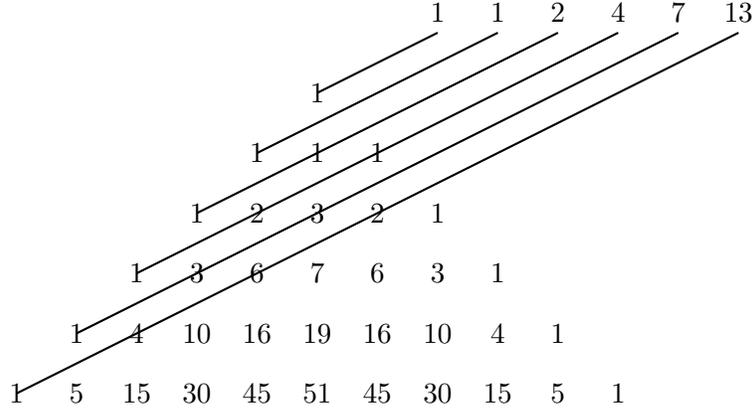
\begin{figure}[ht!]
\begin{center}
\begin{tikzpicture}[scale=1,style=thick]
\begin{scope}[yshift = 0cm, xshift =0cm]
\def\dh{0.8};\def\dv{-0.8};
\node at (0,0) {$1$};

\node  at (-\dh,\dv) {$1$};\node at (0,\dv) {$1$};\node at (\dh,\dv) {$1$};

\node at (-2*\dh,2*\dv)  {$1$};\node at (-\dh,2*\dv)  {$2$};\node at (0,2*\dv)  {$3$};
\node at (\dh,2*\dv)  {$2$};\node at (2*\dh,2*\dv)  {$1$};

\node at (-3*\dh,3*\dv)  {$1$};\node at (-2*\dh,3*\dv)  {$3$};\node at (-\dh,3*\dv)  {$6$};\node at (0,3*\dv)  {$7$};\node at (\dh,3*\dv)  {$6$};\node at (2*\dh,3*\dv)  {$3$};
\node at (3*\dh,3*\dv)  {$1$};

\node at (-4*\dh,4*\dv)  {$1$};\node at (-3*\dh,4*\dv)  {$4$};\node at (-2*\dh,4*\dv)  {$10$};\node at (-\dh,4*\dv)  {$16$};\node at (0,4*\dv)  {$19$};\node at (\dh,4*\dv)  {$16$};\node at (2*\dh,4*\dv)  {$10$};
\node at (3*\dh,4*\dv)  {$4$};\node at (4*\dh,4*\dv)  {$1$};

\node at (-5*\dh,5*\dv)  {$1$};
\node at (-4*\dh,5*\dv)  {$5$};\node at (-3*\dh,5*\dv)  {$15$};\node at (-2*\dh,5*\dv)  {$30$};\node at (-\dh,5*\dv)  {$45$};\node at (0,5*\dv)  {$51$};\node at (\dh,5*\dv)  {$45$};\node at (2*\dh,5*\dv)  {$30$};
\node at (3*\dh,5*\dv)  {$15$};\node at (4*\dh,5*\dv)  {$5$};
\node at (5*\dh,5*\dv)  {$1$};

\draw (-5*\dh,5*\dv) -- (7.0*\dh,-1*\dv);\draw (-4*\dh,4*\dv) -- (6*\dh,-1*\dv);
\draw (-3*\dh,3*\dv) -- (5*\dh,-1*\dv);\draw (-2*\dh,2*\dv) -- (4*\dh,-1*\dv);
\draw (-\dh,\dv) -- (3*\dh,-1*\dv);\draw (0,0) -- (2*\dh,-1*\dv);
\draw (7.0*\dh,-1*\dv) node[above]{$13$};\draw (6.0*\dh,-1*\dv) node[above]{$7$};
\draw (5.0*\dh,-1*\dv) node[above]{$4$};\draw (4.0*\dh,-1*\dv) node[above]{$2$};
\draw (3.0*\dh,-1*\dv) node[above]{$1$};\draw (2.0*\dh,-1*\dv) node[above]{$1$};

\end{scope}
\end{tikzpicture}
\end{center}
\caption{Generalized Pascal triangle of 3-nomial coefficients $\binom{b}{a}_2$ with third order Fibonacci numbers $F^{(3)}_n$ as sums of diagonal entries.\label{fig:pascal}} 
\end{figure}

A {\em composition} of a positive integer $n$ in $k$ parts is a 
way of writing $ n = a_1 + a_2 + \cdots + a_k$ where each summand, called part, is a positive 
integer and the order of the summands is taken into account.
 It is well known, and easy to prove by  stars and bars method, that the number  of compositions of $n\geq1$ in $k$ parts is $\binom{n-1}{k-1}$ and thus the total number of compositions on $n$ is $\sum_{k\geq 1}\binom{n-1}{k-1}= 2^{n-1}$. By convention there exists one composition of 0, a composition in 0 part.

Let $n,k$ with $n\geq k \geq 0$. The coefficient of $x^{n-k}$ in  $(1+x+x^2+\dots +x^{p-1})^{k}$ is the number of ways of writing $n-k$ as sum of $k$ integers $a_i \in \{0,1,2,\dots,p-1\}$.
Increasing each $a_i$ by $1$ we obtain compositions of $n$  with parts in $\{1,2,\dots,p\}$. The following result is thus immediate.
\begin{proposition}
 The number of compositions of  $n$ in $k$ parts belonging to $\{1,2,\dots,p\}$ is the  $p$-nomial coefficient $\binom{k}{n-k}_{p-1}$. 
\end{proposition}
It is also easy to determine the total number of this kind of compositions.
\begin{proposition}
 The total number of compositions of  $n$ in  parts belonging to $\{1,2,\dots,p\}$ is $F_{n+p-1}^{(p)}$. 
\end{proposition}
\begin{proof}
Let $C_n$ be this number.
For $n\leq p $ we have $C_n=F_{n+p-1}^{(p)}$ since the number of compositions of $n$ is $2^{n-1}= F_{n+p-1}^{(p)}$ and all parts are at most $n\leq p$. For $n> p$, considering the possible values of the first part, $C_n= C_{n-1}+C_{n-2}+\dots+ C_{n-p}$ thus by induction $C_n=F_{n+p-1}^{(p)}$.
\end{proof}\qed

\section{Vertices of $p$-th order Fibonacci cubes }
\label{sec:enumprop}

\begin{figure}[ht!]
\begin{center}
\begin{tikzpicture}[scale=0.5,style=thick]
\tikzstyle{every node}=[draw=none,fill=none]
\def\vr{3pt} 
\begin{scope}[yshift = 0cm, xshift = 0cm]

\begin{scope}[yshift = 0cm, xshift = 0cm]
\path (0,0) coordinate (v000);
\path (2.5,0) coordinate (v100);
\path (1.5,1.5) coordinate (v010);
\path (4,1.5) coordinate (v110);
\path (0,2.5) coordinate (v001);
\path (2.5,2.5) coordinate (v101);
\path (1.5,4) coordinate (v011);
\path (4,4) coordinate (v111);
\draw (v000) -- (v100) -- (v110) -- (v010) -- (v000) -- (v001) -- (v011) -- (v111) -- (v101) -- (v100);
\draw (v001) -- (v101); 
\draw (v111) -- (v110); 
\draw (v011) -- (v010); 
\draw (v000) -- (5,-1);
\draw (v100) -- (7.5,-1);
\draw (v010) -- (6.5,0.5);
\draw(v110) -- (9,0.5);
\draw (v001) -- (5,1.5);
\draw (v101) -- (7.5,1.5);
\draw (v011) -- (6.5,3);
\draw (v111) -- (9,3);
\draw (v000)  [fill=white] circle (\vr);
\draw (v100)  [fill=white] circle (\vr);
\draw (v010)  [fill=white] circle (\vr);
\draw (v001)  [fill=white] circle (\vr);
\draw (v110)  [fill=white] circle (\vr);
\draw (v101)  [fill=white] circle (\vr);
\draw (v011)  [fill=white] circle (\vr);
\draw (v111)  [fill=white] circle (\vr);
\draw[left] (v000)++(0.2,0.27) node {\footnotesize $0000$};
\draw[left] (v100)++(0.2,0.27) node {\footnotesize $0100$};
\draw[left] (v010)++(0.2,0.27) node {\footnotesize $0010$};
\draw[left] (v001)++(0.2,0.27) node {\footnotesize $0001$};

\draw[left] (v110)++(0.2,0.27) node {\footnotesize $0110$};
\draw[left] (v101)++(0.2,0.27) node {\footnotesize $0101$};
\draw[left] (v011)++(0.2,0.27) node {\footnotesize $0011$};
\draw[left] (v111)++(0.2,0.27) node {\footnotesize $0111$};

\end{scope}

\begin{scope}[yshift = -1cm, xshift = 5cm]
\path (0,0) coordinate (v000);
\path (2.5,0) coordinate (v100);
\path (1.5,1.5) coordinate (v010);
\path (4,1.5) coordinate (v110);
\path (0,2.5) coordinate (v001);
\path (2.5,2.5) coordinate (v101);
\path (1.5,4) coordinate (v011);
\path (4,4) coordinate (v111);
\draw (v000) -- (v100) -- (v110) -- (v010) -- (v000) -- (v001) -- (v011) -- (v111) -- (v101) -- (v100);
\draw (v001) -- (v101); 
\draw (v111) -- (v110); 
\draw (v011) -- (v010); 
\draw (v000)  [fill=white] circle (\vr);
\draw (v100)  [fill=white] circle (\vr);
\draw (v010)  [fill=white] circle (\vr);
\draw (v001)  [fill=white] circle (\vr);
\draw (v110)  [fill=white] circle (\vr);
\draw (v101)  [fill=white] circle (\vr);
\draw (v011)  [fill=white] circle (\vr);
\draw (v111)  [fill=white] circle (\vr);
\draw[right] (v000)++(-0.2,-0.27) node {\footnotesize $1000$};
\draw[right] (v100)++(-0.2,-0.27) node {\footnotesize $1100$};
\draw[right] (v010)++(-0.2,-0.27) node {\footnotesize $1010$};
\draw[right] (v001)++(-0.2,-0.27) node {\footnotesize $1001$};

\draw[right] (v110)++(-0.2,-0.27) node {\footnotesize $1110$};
\draw[right] (v101)++(-0.2,-0.27) node {\footnotesize $1101$};
\draw[right] (v011)++(-0.2,-0.27) node {\footnotesize $1011$};
\draw[right] (v111)++(-0.2,-0.27) node {\footnotesize $1111$};

\end{scope}
\end{scope}
\begin{scope}[yshift = 0cm, xshift = -15cm]
\begin{scope}[yshift = 0cm, xshift = 0cm]
\path (0,0) coordinate (v000);
\path (2.5,0) coordinate (v100);
\path (1.5,1.5) coordinate (v010);
\path (4,1.5) coordinate (v110);
\path (0,2.5) coordinate (v001);
\path (2.5,2.5) coordinate (v101);
\path (1.5,4) coordinate (v011);
\path (4,4) coordinate (v111);
\draw (v000) -- (v100) -- (v110) -- (v010) -- (v000) -- (v001) -- (v011) -- (v111) -- (v101) -- (v100);
\draw (v001) -- (v101); 
\draw (v111) -- (v110); 
\draw (v011) -- (v010); 
\draw (v000) -- (5,-1);
\draw (v100) -- (7.5,-1);
\draw (v010) -- (6.5,0.5);
\draw(v110) -- (9,0.5);
\draw (v001) -- (5,1.5);
\draw (v101) -- (7.5,1.5);
\draw (v011) -- (6.5,3);
\draw (v000)  [fill=white] circle (\vr);
\draw (v100)  [fill=white] circle (\vr);
\draw (v010)  [fill=white] circle (\vr);
\draw (v001)  [fill=white] circle (\vr);
\draw (v110)  [fill=white] circle (\vr);
\draw (v101)  [fill=white] circle (\vr);
\draw (v011)  [fill=white] circle (\vr);
\draw (v111)  [fill=white] circle (\vr);
\draw[left] (v000)++(0.2,0.27) node {\footnotesize $0000$};
\draw[left] (v100)++(0.2,0.27) node {\footnotesize $0100$};
\draw[left] (v010)++(0.2,0.27) node {\footnotesize $0010$};
\draw[left] (v001)++(0.2,0.27) node {\footnotesize $0001$};

\draw[left] (v110)++(0.2,0.27) node {\footnotesize $0110$};
\draw[left] (v101)++(0.2,0.27) node {\footnotesize $0101$};
\draw[left] (v011)++(0.2,0.27) node {\footnotesize $0011$};
\draw[left] (v111)++(0.2,0.27) node {\footnotesize $0111$};

\end{scope}

\begin{scope}[yshift = -1cm, xshift = 5cm]
\path (0,0) coordinate (v000);
\path (2.5,0) coordinate (v100);
\path (1.5,1.5) coordinate (v010);
\path (4,1.5) coordinate (v110);
\path (0,2.5) coordinate (v001);
\path (2.5,2.5) coordinate (v101);
\path (1.5,4) coordinate (v011);
\draw (v000) -- (v100) -- (v110) -- (v010) -- (v000) -- (v001) -- (v011); draw (v101) -- (v100);
\draw (v001) -- (v101); \draw (v101) -- (v100); 

\draw (v011) -- (v010); 
\draw (v000)  [fill=white] circle (\vr);
\draw (v100)  [fill=white] circle (\vr);
\draw (v010)  [fill=white] circle (\vr);
\draw (v001)  [fill=white] circle (\vr);
\draw (v110)  [fill=white] circle (\vr);
\draw (v101)  [fill=white] circle (\vr);
\draw (v011)  [fill=white] circle (\vr);
\draw[right] (v000)++(-0.2,-0.27) node {\footnotesize $1000$};
\draw[right] (v100)++(-0.2,-0.27) node {\footnotesize $1100$};
\draw[right] (v010)++(-0.2,-0.27) node {\footnotesize $1010$};
\draw[right] (v001)++(-0.2,-0.27) node {\footnotesize $1001$};

\draw[right] (v110)++(-0.2,-0.27) node {\footnotesize $1110$};
\draw[right] (v101)++(-0.2,-0.27) node {\footnotesize $1101$};
\draw[right] (v011)++(-0.2,-0.27) node {\footnotesize $1011$};

\end{scope}
\end{scope}

\begin{scope}[yshift = 10cm, xshift = 0cm]
\begin{scope}[yshift = 0cm, xshift = 0cm]
\path (0,0) coordinate (v000);
\path (2.5,0) coordinate (v100);
\path (1.5,1.5) coordinate (v010);
\path (4,1.5) coordinate (v110);
\path (0,2.5) coordinate (v001);
\path (2.5,2.5) coordinate (v101);
\path (1.5,4) coordinate (v011);
\draw (v000) -- (v100) -- (v110) -- (v010) -- (v000) -- (v001) -- (v011); \draw (v101) -- (v100);
\draw (v001) -- (v101); 
\draw (v011) -- (v010); 
\draw (v000) -- (5,-1);
\draw (v100) -- (7.5,-1);
\draw (v010) -- (6.5,0.5);
\draw (v001) -- (5,1.5);
\draw (v101) -- (7.5,1.5);
\draw (v011) -- (6.5,3);
\draw (v000)  [fill=white] circle (\vr);
\draw (v100)  [fill=white] circle (\vr);
\draw (v010)  [fill=white] circle (\vr);
\draw (v001)  [fill=white] circle (\vr);
\draw (v110)  [fill=white] circle (\vr);
\draw (v101)  [fill=white] circle (\vr);
\draw (v011)  [fill=white] circle (\vr);
\draw[left] (v000)++(0.2,0.27) node {\footnotesize $0000$};
\draw[left] (v100)++(0.2,0.27) node {\footnotesize $0100$};
\draw[left] (v010)++(0.2,0.27) node {\footnotesize $0010$};
\draw[left] (v001)++(0.2,0.27) node {\footnotesize $0001$};

\draw[left] (v110)++(0.2,0.27) node {\footnotesize $0110$};
\draw[left] (v101)++(0.2,0.27) node {\footnotesize $0101$};
\draw[left] (v011)++(0.2,0.27) node {\footnotesize $0011$};

\end{scope}

\begin{scope}[yshift = -1cm, xshift = 5cm]
\path (0,0) coordinate (v000);
\path (2.5,0) coordinate (v100);
\path (1.5,1.5) coordinate (v010);
\path (0,2.5) coordinate (v001);
\path (2.5,2.5) coordinate (v101);
\path (1.5,4) coordinate (v011);
\draw (v000) -- (v100); \draw (v010) -- (v000) -- (v001) -- (v011); \draw (v101) -- (v100);
\draw (v001) -- (v101); \draw (v101) -- (v100); 

\draw (v011) -- (v010); 
\draw (v000)  [fill=white] circle (\vr);
\draw (v100)  [fill=white] circle (\vr);
\draw (v010)  [fill=white] circle (\vr);
\draw (v001)  [fill=white] circle (\vr);
\draw (v101)  [fill=white] circle (\vr);
\draw (v011)  [fill=white] circle (\vr);
\draw[right] (v000)++(-0.2,-0.27) node {\footnotesize $1000$};
\draw[right] (v100)++(-0.2,-0.27) node {\footnotesize $1100$};
\draw[right] (v010)++(-0.2,-0.27) node {\footnotesize $1010$};
\draw[right] (v001)++(-0.2,-0.27) node {\footnotesize $1001$};

\draw[right] (v101)++(-0.2,-0.27) node {\footnotesize $1101$};
\draw[right] (v011)++(-0.2,-0.27) node {\footnotesize $1011$};

\end{scope}
\end{scope}

\begin{scope}[yshift = 10cm, xshift = -15cm]
\begin{scope}[yshift = 0cm, xshift = 0cm]
\path (0,0) coordinate (v000);
\path (2.5,0) coordinate (v100);
\path (1.5,1.5) coordinate (v010);
\path (0,2.5) coordinate (v001);
\path (2.5,2.5) coordinate (v101);
\draw (v000) -- (v100); \draw(v010) -- (v000) -- (v001); \draw(v101) -- (v100);
\draw (v001) -- (v101); 
\draw (v000) -- (5,-1);
\draw (v010) -- (6.5,0.5);
\draw (v001) -- (5,1.5);
\draw (v000)  [fill=white] circle (\vr);
\draw (v100)  [fill=white] circle (\vr);
\draw (v010)  [fill=white] circle (\vr);
\draw (v001)  [fill=white] circle (\vr);
\draw (v101)  [fill=white] circle (\vr);
\draw[left] (v000)++(0.2,0.27) node {\footnotesize $0000$};
\draw[left] (v100)++(0.2,0.27) node {\footnotesize $0100$};
\draw[left] (v010)++(0.2,0.27) node {\footnotesize $0010$};
\draw[left] (v001)++(0.2,0.27) node {\footnotesize $0001$};

\draw[left] (v101)++(0.2,0.27) node {\footnotesize $0101$};

\end{scope}

\begin{scope}[yshift = -1cm, xshift = 5cm]
\path (0,0) coordinate (v000);
\path (1.5,1.5) coordinate (v010);
\path (0,2.5) coordinate (v001);
 \draw (v010) -- (v000) -- (v001); 
%
\draw (v000)  [fill=white] circle (\vr);
\draw (v010)  [fill=white] circle (\vr);
\draw (v001)  [fill=white] circle (\vr);
\draw[right] (v000)++(-0.2,-0.27) node {\footnotesize $1000$};
\draw[right] (v010)++(-0.2,-0.27) node {\footnotesize $1010$};
\draw[right] (v001)++(-0.2,-0.27) node {\footnotesize $1001$};


\end{scope}
\end{scope}

\end{tikzpicture}
\end{center}
\caption{$\Gamma_4^{(2)}=\Gamma_4$, $\Gamma_4^{(3)}$, $\Gamma_4^{(4)}$ and $Q_4$  }
\label{fig:Gamma 4}
\end{figure}
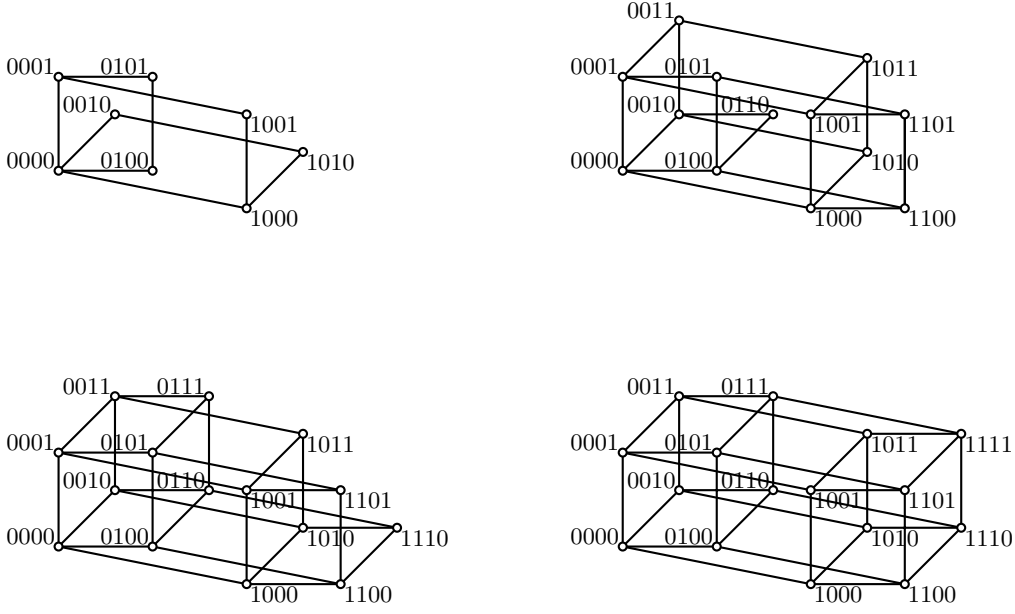

We will generalize the well-known properties  that the order of $\Gamma_n$ is $F_{n+2}$ and that the number of vertices with weight $w$ is $\binom{n-w+1}{w}$. For this purpose we will construct a bijection between $p$-th order Fibonacci strings of length $n$ and weight $w$ 
and compositions of $n+1$ in $n+1-w$ parts belonging  to $\{1,2,\dots,p\}$.

\begin{proposition}
 The number of vertices of $\Gamma_n^{(p)}$ is $F_{n+p}^{(p)}$. 
\end{proposition}

\begin{proof} 
This is true for $0\leq n\leq p-1$ since $\Gamma_n^{(p)}=Q_n$ and $F_{n+p}^{(p)}=2^n$ because $n+p\in[[p,2p-1]]$.
Now assume $n\geq p$. A string in  ${\cal F}^{(p)}_n$ can start with a $0$ or with $1^i0$ where $1\leq i \leq p-1$. Therefore we have the partition ${\cal F}^{(p)}_n=\bigcup_{i=0}^{p-1}{1^i0{\cal F}^{(p)}_{n-i-1}}$ and thus $|{\cal F}^{(p)}_n|=\sum_{i=1}^p{|{\cal F}^{(p)}_{n-i}|}$. By induction the result follows.
\end{proof}\qed

From the expression (\ref{eq:gen(p)}) of the generating function of the sequence $(F_{n}^{(p)})_{n\geq 0}$ we obtain
\begin{corollary}
 The generating function of the order of $\Gamma_n^{(p)}$ is 
\begin{equation*}
\sum_{n\geq 0}{|V(\Gamma^{(p)}_{n}})|t^n=t^{-1}\left(\frac{1}{1-t-t^2-\dots-t^{p}}-1\right)=\frac{1+t+\dots+t^{p-1}}{1-t-t^2-\dots-t^{p}}.
\end{equation*}
\end{corollary}
 \begin{theorem} \label{th:weight}
 The number of vertices of $\Gamma_n^{(p)}$ of weight $w$ is $\binom{n-w+1}{w}_{p-1}$. 
\end{theorem}
\begin{proof} 
Let $u\in {\cal F}^{(p)}_n$
with $w(u)= w$. Let $E$ be the set of strings $E=\{e_1,e_2,\dots,e_p\}$ where $e_1=0$ and $e_i=1^{i-1}0$ for $i\in [\![2,p]\!] .$ 
In the string $u0$, concatenation of $u$ with 0, every sequence of $1$s is followed by at least one $0$. 
Therefore $u0$ can be decomposed as concatenation of strings in $E$ and this decomposition is unique.
Consider the mapping from $E$ to $\{1,2,\dots,p\}$ defined by $f(e_i)=i=w(e_i)+1$. Extending this mapping to concatenation of strings of $E$ we  can associate to $u0$  a list of integers $a_1,a_2\dots, a_k$ in $ [\![1,p]\!] $ with $\sum_{i=1}^k{(a_i-1)}=w(u0)=w$ and  $\sum_{i=1}^k{a_i}=n+1$. From this two relations $k=n+1-w$ and thus $n+1=a_1+a_2+\dots+ a_k$ defines a composition of $n+1$ in $n-w+1$ parts belonging to $\{1,2,\dots,p\}$. Il is clear that the process can be reversed, starting from such a composition $n+1=a_1+a_2+\dots+ a_{n-w+1}$ the string $f^{-1}(a_1)f^{-1}(a_2)\dots f^{-1}(a_{n-w+1})$ is $v0$ where $v$ is a string of ${\cal F}^{(p)}_n$ with weight $w$.
\end{proof}\qed

As a corollary we obtain a combinatorial proof of the equality
\begin{equation*}
F_n^{(p)}=\sum_{w=0}^{\left\lfloor \frac{(n+1)(p-1)}{p} \right\rfloor}{\binom{n-w+1}{w}_{p-1}}\,.
\end{equation*}
Therefore $p$-th order Fibonacci numbers appear as sum of diagonal entries in generalized Pascal triangle of $p$-nomial coefficients (Figure~\ref{fig:pascal}).

Since $\binom{b}{a}_{p-1} >0$ is and only if $a\leq b(p-1)$, another immediate consequence of Theorem~\ref{th:weight} is the following proposition.
\begin{proposition} 
The maximum weight of a vertex of $\Gamma_n^{(p)}$ is $\left\lfloor \frac{(n+1)(p-1)}{p}\right\rfloor$.
\end{proposition}
\section{Edges in $p$-th order Fibonacci cubes }
\label{sec:edge}

The recurrence structure of $p$-th order generalized Fibonacci cubes is the following {\em fundamental decomposition}.

Consider the partition ${\cal F}^{(p)}_n=\bigcup_{i=1}^{p}{1^{i-1}0{\cal F}^{(p)}_{n-i}}$ of the vertex set of $\Gamma_n^{(p)}$ for $n\geq p$. From this partition there exist two kind of edges $uv$ in $\Gamma_n^{(p)}$. Those such that $u,v$ belong to the same ${1^{i-1}0{\cal F}^{(p)}_{n-i}}$ and those where $u,v$ are in different sets.

For $i\in [\![1,p]\!]$ the set $1^{i-1}0{\cal F}^{(p)}_{n-i}$ induced a subgraph of  $\Gamma_n^{(p)}$  isomorphic to $\Gamma_{n-i}^{(p)}$. Moreover for $i,j\in [\![1,p]\!]$ with $j<i$ any vertex $1^{i-1}0x$ has exactly one neighbor in $1^{j-1}0\Gamma_{n-j}^{(p)}$ which is the vertex $1^{j-1}01^{i-j-1}0x$.\\ 

Therefore the two kind of edges are:
\begin{itemize}
\item For all $i\in [\![1,p]\!]$ edges from a subgraph isomorphic to $\Gamma_{n-i}^{(p)}$. 

\item For all $i,j\in [\![1,p]\!]$  with $j<i$ edges from a matching between  ${1^{i-1}0{\cal F}^{(p)}_{n-i}}$ and ${1^{j-1}01^{i-j-1}{\cal F}^{(p)}_{n-j}}$.
\end{itemize}

Since there exist $i-1$ integers $j$ in $[\![1,p]\!]$ with $j<i$, a vertex in ${1^{i-1}0{\cal F}^{(p)}_{n-i}}$ is adjacent to exactly $i-1$ vertices in some ${1^{j-1}0{\cal F}^{(p)}_{n-i}}$ with $j<i$. Therefore, counting the edges by there endpoint of maximum weight, the number of edges of the second kind is $\sum_{i=2}^{p}(i-1){|V(\Gamma^{(p)}_{n-i})|}$ and we deduce the following result.

\begin{proposition}
 For $n\geq p$
\begin{equation}\label{eq:recedge}
|E(\Gamma^{(p)}_{n})|=\sum_{i=1}^{p}{|E(\Gamma^{(p)}_{n-i})|}+\sum_{i=2}^{p}(i-1){|V(\Gamma^{(p)}_{n-i})|}\,.
\end{equation}
\end{proposition}

\begin{theorem}\label{th:genesize}
 The generating function of the size of $\Gamma_n^{(p)}$ is 
\begin{equation*}
\sum_{n\geq 0}{|E(\Gamma^{(p)}_{n}})|t^n=\frac{t+2t^2+\dots+(p-1)t^{p-1}}{(1-t-t^2-\dots-t^{p})^2}.
\end{equation*}
\end{theorem}
\begin{proof}
It will be convenient to introduce the following notations. For an integer $i\in\ZZ$ define $m_i=0$ for $i<0$ and $m_i=|E(\Gamma^{(p)}_{i})|$ otherwise. Furthermore, let us set $v_i=F_{i+p}^{(p)}$ for $i\geq -p$. Note that $v_i=0$ for $i\leq -2$, $v_{-1}=1$ and $v_i=|V(\Gamma^{(p)}_{i})|$ for $i\geq 0$.

Let $n\geq p$ then since $n-i\geq 0$ for $i\leq p$, the identity~(\ref{eq:recedge}) can be rewritten 
\begin{equation}\label{eq:withournotation}
m_n=\sum_{i=1}^{p}{m_{n-i}}+\sum_{i=1}^{p-1}{i\,v_{n-i-1}}\,.
\end{equation}
Assume now $1\leq n< p$. The partition of the vertex set becomes $${\cal F}^{(p)}_n=\bigcup_{i=1}^{n}{1^{i-1}0{\cal F}^{(p)}_{n-i}}\bigcup \{1^n\}\,.$$
The vertex $1^n$ is adjacent to the $n$ vertices belonging to  $\{1^{i-1}01^{n-i}\st 1\leq i \leq n \}$.

We have thus the relation 
\begin{equation*}
|E(\Gamma^{(p)}_{n})|=\sum_{i=1}^{n}{|E(\Gamma^{(p)}_{n-i})|}+\sum_{i=2}^{n}(i-1){|V(\Gamma^{(p)}_{n-i})|}+n\,.
\end{equation*} 
or equivalently
\begin{equation*}
m_n=\sum_{i=1}^{n}{m_{n-i}}+\sum_{i=1}^{n-1}{i\,v_{n-i-1}}+n\,v_{-1}\,.
\end{equation*}
Therefore, since $\sum_{i=n+1}^{p}m_{n-i}=0$ and $\sum_{i=n+1}^{p-1}{i\,v_{n-i-1}}=0$,  relation (\ref{eq:withournotation})  is also satisfied for $n< p$.

Multiplying identity (\ref{eq:withournotation}) by $t^n$ and summing for $n\geq 0$ we obtain
\begin{eqnarray*}
\sum_{n\geq 0}t^n m_n&=&\sum_{n\geq 0}t^n{\sum_{i=1}^{p}{m_{n-i}}}+\sum_{n\geq 0}t^n{\sum_{i=1}^{p-1}{i\,v_{n-i-1}}}\\
&=&\sum_{i=1}^{p}{t^i\sum_{n\geq 0}t^{n-i}{m_{n-i}}}+\sum_{i=1}^{p-1}{i\,t^i\sum_{n\geq 0}{t^{n-i}v_{n-i-1}}}\\
&=&\sum_{i=1}^{p}{t^i\sum_{n\geq 0}t^{n}{m_{n}}}+\sum_{i=1}^{p-1}{i\,t^i\sum_{n\geq 0}{t^{n}v_{n-1}}}\,.
\end{eqnarray*}
From here we obtain
\begin{equation*}
(1-\sum_{i=1}^{p}{t^i})\sum_{n\geq 0}t^n m_n=(\sum_{i=1}^{p-1}{i\,t^i})\sum_{n\geq 0}{t^{n}v_{n-1}}\,.
\end{equation*} 
From the generating function of the sequence $(F_m^{(k)})_{m\geq0}$
\begin{equation*}
\sum_{n\geq 0}{t^{n}v_{n-1}}=\sum_{n\geq 1-p}{t^{n}v_{n-1}}=\sum_{n\geq 1-p}{t^{n}F_{n+p-1}}=t^{1-p}\sum_{m\geq0}{F_mt^m}=(1-\sum_{i=1}^{p}{t^i})^{-1}\,.
\end{equation*}
Therefore $\sum_{n\geq 0}t^n m_n=(\sum_{i=1}^{p-1}{i\,t^i})(1-\sum_{i=1}^{p}{t^i})^{-2}$.
This completes the proof.
\end{proof}\qed
%

\section{Cube polynomial of $p$-th order Fibonacci cubes }
\label{sec:distancecube}
This section is devoted to the determination of the generating function of the cube polynomial using the weight enumerator polynomial.

Let $\A$ be a set of strings generated freely (as a monoid) by a finite or  infinite alphabet
$E$. 
That means that every strings $s \in \A$ can be written 
uniquely as a concatenation of zero or more strings from $E$. Note that the empty string $\lambda$ belongs to $\A$.
Associate with any string $s \in \A$ a polynomial, possibly multivariate,  $\theta_s$. We will say that $s \rightarrow \theta_s$ is \emph{concatenation multiplicative} over $\A$ if $\theta_{\lambda}=1$ and $\theta_{s_1s_2\ldots s_k}=\theta_{s_1}\theta_{s_2}\ldots\theta_{s_k}$ for any choice of elements $s_i$ in $E$.

Note that if $s \rightarrow \theta_s$ and $s \rightarrow \phi_s$ are concatenation multiplicative then $s \rightarrow \theta_s\phi_s$ is concatenation multiplicative.
Let us recall the following proposition proved in~\cite{M-2025a}.

\begin{proposition}\label{pro:mono}

Let $\A$ be a set of strings generated freely, as a monoid, by the alphabet $E$ and $s \rightarrow \theta_s$ concatenation multiplicative over $\A$ then
\begin{equation*}
\sum_{s\in \A}{\theta_s}=\frac{1}{1-\sum_{s\in E}{\theta_s}}\,.
\end{equation*}
\end{proposition}

\begin{definition}
Let $G$ be a subgraph of $Q_n$ and $w_d(G)$, $d\ge 0$, be the number of vertices of $G$ with weight $d$. Then the weight enumerator polynomial of $G$ is the counting polynomial
$$W_{G}(x)=\sum_{u\in V(G)} x^{w(u)}= \sum_{d\geq 0} w_{d}(G) x^d\,.$$
\end{definition}

Let us call {\em extended  $p$-th order Fibonacci strings} the strings of $\bigcup_{n\geq 0}\{u0\st u\in{\cal F}^{(p)}_n\}$ obtained by adding $0$ to $p$-th order Fibonacci strings.
As noticed in Section~\ref{sec:enumprop}
 extended  $p$-th order Fibonacci strings can be uniquely obtained as concatenations of strings from $E=\{e_1,e_2,\dots,e_p\}$ where $e_1=0$ and $e_i=1^{i-1}0$ for $i\in [\![2,p]\!]$. The monoid generated freely by $E$ is thus
\begin{equation}\label{eq:partition}
\A= \{\lambda\}\cup_{n\geq 0}\{u0\st u\in{\cal F}^{(p)}_n\}.
\end{equation}

Associate with any string binary $s$ the polynomial $\theta_s(x,t)=x^{w(s)}t^{l(s)}$ where $l(s)$ and $w(s)$ are the length and weight of the string, respectively. It is immediate that  $s \rightarrow \theta_s$ is concatenation multiplicative over $\A$ and that $\theta_{s0}=t\theta_{s}$.  Since   $\sum_{s\in E}{\theta_s}=t+xt^2+\dots+x^{p-1}t^{p}$, using  Proposition~\ref{pro:mono}, we obtain 
\begin{equation}\label{eq:somoverA}
\sum_{s\in \A}{\theta_s}=\frac{1}{1-t(1+xt+\dots+x^{p-1}t^{p-1})}\,.
\end{equation}
From the partition~(\ref{eq:partition})  and the equality~(\ref{eq:somoverA}) we deduce
\begin{equation*}
\sum_{n\geq0}\sum_{s\in{\cal F}^{(p)}_n}t^p{\theta_s}=\frac{1}{t}\left(\frac{1}{1-t(1+xt+\dots+x^{p-1}t^{p-1})}-1\right)\,.
\end{equation*}
Therefore the generating function of $W_{ \Gamma^{(p)}_n}(x)$ is given by the following result.
\begin{theorem} \label{thm:weightfibgene}
The generating function of $W_{\Gamma_n^{(p)}}(x)$ is
\begin{equation*}
\sum_{n\geq0}{W_{\Gamma_n^{(p)}}(x)t^n}=\frac{1+xt+x^2t^2+\dots+x^{p-1}t^{p-1}}{1-t(1+xt++x^2t^2\dots+x^{p-1}t^{p-1})}\,.
\end{equation*}
\end{theorem}

Note that Theorem~\ref{th:weight} can be also deduced from Theorem~\ref{thm:weightfibgene}. Indeed we can rewrite the generating function of $W_{\Gamma_n^{(p)}}(x)$
\begin{eqnarray*}
\sum_{n\geq0}{W_{\Gamma_n^{(p)}}(x)t^n}& = &(1+xt+\ldots+x^{p-1}t^{p-1})\sum_{a\geq0}{(t(1+xt+\ldots+x^{p-1}t^{p-1}))^a}\\
& = &\sum_{a\geq0}t^a(1+xt+\ldots+x^{p-1}t^{p-1})^{a+1}\,.
\end{eqnarray*}
The number of vertices with weight $w$ in $\Gamma_n^{(p)}$ is the coefficient of $t^nx^w$ in $\sum_{n\geq0}{W_{\Gamma_n^{(p)}}(x)t^n}$.
The expansion of $(1+xt+\dots+x^{p-1}t^{p-1})^{a+1}$ produces only monomials $t^bx^c$ with $b=c$ therefore they  contribute to $t^nx^w$ if and only if $a=n-w$ and $b=w$. Furthermore this contribution is  $\binom{a+1}{b}_{p-1}=\binom{n-w+1}{w}_{p-1}$ and we obtain the following alternate expression of the weight polynomial
\begin{equation}\label{eq:walt}
W_{\Gamma_n^{(p)}}(x)=\sum_{w=0}^{\left\lfloor \frac{(n+1)(p-1)}{p}\right\rfloor}\binom{n-w+1}{w}_{p-1}x^w\,.
\end{equation}

Since changing a $1$ to $0$ in a $p$-th order Fibonacci string gives a $p$-th order Fibonacci string it is clear, as noticed in~\cite {KM-2019a}, that $p$-th order generalized Fibonacci cubes are daisy cubes generated by the set of maximal $p$-th order Fibonacci strings.

if $G$ is a daisy cube, then the polynomials $D_{G}$ and $C_{G}$ are completely determined by the weight polynomial. More precisely it is proved in~\cite[Corollary 2.6]{KM-2019a} that for a daisy cube
\begin{equation*}
D_G(x,q)=C_G(x+q-1)=W_G(x+q)\,.
\end{equation*}

Therefore replacing $x$ by $x+1$ in the generating function of the weight polynomial obtained in Theorem~\ref{thm:weightfibgene} gives the generating function of the cube polynomial.

We deduce also that $D_{\Gamma_n^{(p)}}(x,q)=D_{\Gamma_n^{(p)}}(q,x)$  and replacing $x$ by $x+q$ in the generating function of $ W_{\Gamma_n^{(p)}}(x)$ gives that of the distance cube polynomial. 
\begin{corollary}\label{co:cubepolyfibgene}
The generating function of $C_{\Gamma_n^{(p)}}(x)$ 	and $D_{\Gamma_n^{(p)}}(x,q)$ are
\begin{eqnarray*}
\sum_{n\geq 0}C_{\Gamma_n^{(p)}}(x)t^n&=&\frac{1+(x+1)t+(x+1)^2t^2+\ldots+(x+1)^{p-1}t^{p-1}}{1-t(1+(x+1)t+(x+1)^2t^2+\ldots+ (x+1)^{p-1}t^{p-1})}\\
\sum_{n\geq 0}D_{\Gamma_n^{(p)}}(x,q)t^n&=&\frac{1+(x+q)t+(x+q)^2t^2+\ldots+(x+q)^{p-1}t^{p-1}}{1-t(1+(x+q)t+(x+q)^2t^2+\ldots+ (x+q)^{p-1}t^{p-1})}\,.
\end{eqnarray*}
\end{corollary}
Starting from the alternate expression~(\ref{eq:walt}) of the weight polynomial we obtain an other expression of the cube polynomial.

\begin{theorem} \label{th:ccube}
If $n\ge 0$, then  $C_{\Gamma_n^{(p)}}(x)$ is of degree  $\left\lfloor \frac{(n+1)(p-1)}{p}\right\rfloor$. Moreover
\begin{equation*}
C_{\Gamma_n^{(p)}}(x)=\sum_{a = 0}^{\left\lfloor \frac{(n+1)(p-1)}{p}\right\rfloor}\binom{n-a+1}{a}_{p-1}(1+x)^{a}
\end{equation*} and
\begin{equation*}
D_{\Gamma_n^{(p)}}(x)=\sum_{a = 0}^{\left\lfloor \frac{(n+1)(p-1)}{p}\right\rfloor}\binom{n-a+1}{a}_{p-1}(x+q)^{a}\,.
\end{equation*}
The number of induced subgraphs of $\Gamma_n^{(p)}$ isomorphic to  $Q_k$  is 
\begin{equation*}
c_{k}(\Gamma_n^{(p)})=\sum_{a = k}^{\left\lfloor \frac{(n+1)(p-1)}{p}\right\rfloor}\binom{n-a+1}{a}_{p-1} \binom{a}{k}
\end{equation*}
and the number of those at distance $d$ of $0^n$ is
\begin{equation*}
  c_{k,d}(\Gamma_n^{(p)})=\binom{n-d-k+1}{d+k}_{p-1} \binom{d+k}{k}\,.
\end{equation*}
\end{theorem}
\begin{proof}
It is straightforward to derive the expression of $c_{k}(\Gamma_n^{(p)})$ from that of $C_{\Gamma_n^{(p)}}(x)$. The same arguments applied to the distance cube polynomial give the expression for $c_{k,d}(\Gamma_n^{(p)})$. The value of $c_{k}(\Gamma_n^{(p)})$ can also be obtained by combinatorial arguments by counting the induced $k$-cubes by their top vertex. Indeed the top vertex $u$ of a subgraph isomorphic to a $Q_k$ satisfies $w(u)\geq k$. There exist $\binom{n-a+1}{a}_{p-1}$ vertices of weight $a\geq k$. Choose a set $K$ of $k$ coordinates among the $a$ such that $u_i=1$. A $Q_k$ with top vertex $u$ is induced by the $2^k$ vertices obtained by varying these $k$ coordinates. Therefore $u$ is the top vertex of $\binom{a}{k}$ different induced $k$-cubes. Since an induced $k$-cube with a top vertex of weight $a$ is at distance $a-k$ of $0^n$ the expressions for $c_{k}(\Gamma_n^{(p)})$ and $c_{k,d}(\Gamma_n^{(p)})$ follow.
\end{proof}\qed

Let $k\geq0$. Since $c_{k}(\Gamma_n^{(p)})=\sum_{w\geq k}\binom{n-w+1}{w}_{p-1} \binom{w}{k}$ the generating sequence of the sequence $ (c_{k}(\Gamma_n^{(p)}))_{n\geq 0}$ can be expressed from the generating function of the weight polynomial determined in Theorem~\ref{thm:weightfibgene}.
Indeed consider the weight polynomial $W_{\Gamma_n^{(p)}}=g(x)=\sum_{w\geq 0}\binom{n-w+1}{w}_{p-1}\,x^w$.
The $k$-th derivative of $g(x)$ is $g^{[k]}(x)=\sum_{w\geq k}\binom{n-w+1}{w}_{p-1}\,\frac{w!}{(w-k)!}\,x^{w-k}$ and therefore $c_{k}(\Gamma_n^{(p)})$ is the value in $x=1$ of $\frac{1}{k!}g^{[k]}(x)$. We have thus the following result that give, for example, another proof of the expressions of the generating function of the size of $\Gamma_n^{(p)}$ in Theorem~\ref{th:genesize}. 

\begin{corollary}
For $k\geq 0$ the generating function of the sequence $(c_{k}(\Gamma_n^{(p)})_{n\geq0}$ of 
the number of induced subgraphs of $\Gamma_n^{(p)}$ isomorphic to  $Q_k$  is 
\begin{equation*}
\sum_{n\geq0}{c_{k}(\Gamma_n^{(p)})}t^n= \frac{1}{k!}h(1,t)
\,,
\end{equation*}
where 
\begin{equation*}
h(x,t)=\frac{\partial^k}{\partial x^k} \left(\frac{1+xt+x^2t^2+\dots+x^{p-1}t^{p-1}}{1-t(1+xt+x^2t^2\dots+x^{p-1}t^{p-1})}\right)\,.
\end{equation*}
\end{corollary}

%

\section{Maximal hypercubes in Fibonacci $p$-cubes}
\label{sec:maximalcube}

Is is proved  in \cite {M-2012a} that the number of maximal hypercubes of dimension $k$ in $\Gamma_n$ is $$h_k(\Gamma_n)=\binom{k+1}{n-2k+1}\,.$$
This number is thus the number of vertices of weight $n-2k+1$ in $\Gamma_{n-k+1}$ and indeed this result is obtained by an explicit one to one mapping between top vertices of maximal hypercubes of dimension $k$ in $\Gamma_n$ and
Fibonacci strings of length $n-k+1$ with Hamming weight $n-2k+1$. Let us recall that Fibonacci strings can be seen as Fibonacci $1$-strings or $2$nd order Fibonacci strings.

In this section we will generalize this result, for $p\geq1$, exhibiting a bijection  between maximal hypercubes in Fibonacci $p$-cubes $\Gamma^{p}_n$ and vertices of weight $n-(p+1)k+p$ in the $p+1$-th order Fibonacci cube $\Gamma^{(p+1)}_{n-pk+p}$.

 Let $H$ be an induced subgraph of $Q_n$ isomorphic to some $Q_k$. For $i\notin \supp(H)$, 
we will denote by $H\widetilde{+}\delta_i$ the subgraph induced by 
$V(H)\cup\{x+\delta_i \st x\in V(H)\}$. Note that $H\widetilde{+}\delta_i$ is isomorphic to $Q_{k+1}$.

Let us thus first characterize top vertices of maximal hypercubes in $\Gamma_n^p$.

\begin{lemma}\label{lemma:vertexmax} 
 If $H$ is a maximal hypercube of dimension $k$ in $\Gamma_n^p$, then $b(H)=0^n$ and $t(H)=0^{l_0}10^{l_1}\dots 10^{l_i}\dots10^{l_{k}}$, where $\sum_{i=0}^k{l_i}=n-k$, $l_0, l_k \in [[0,p]]$, and $l_i\in [[p,2p]]$ for $i \in [[1,k-1]]$.
Furthermore, any such vertex $t(H)$ is the top vertex of a unique maximal hypercube.
 \end{lemma}
 \begin{proof}
 Let $H$ be a maximal hypercube in $\Gamma_n^p$.
 Assume there exists an integer  $j$ such that  $b(H)_j = 1$. 
Then $H\widetilde{+}\delta_j$ must be an induced subgraph of $\Gamma_n^p$, which contradicts the 
maximality of $H$. 
   
 Therefore $t(H)$ is a Fibonacci $p$-string of weight $k$  and thus $t(H)=0^{l_0}10^{l_1}\dots 10^{l_i}\dots10^{l_{k}}$ with $l_0, l_k\geq0$ and $l_i\geq p$ for $i \in [[1,k-1]]$ where $\sum_{i=0}^k{l_i}=n-k$.
 
Assume that $l_0\geq p+1$. Then for any vertex $x$ of $H$ obviously $x+\delta_1$ $\in V(\Gamma_n^p)$.  Therefore  $H$ is contained in $H\widetilde{+}\delta_1$ an induced $Q_{k+1}$ of $\Gamma_n^p$, which again
contradicts the maximality of $H$. The case  $l_k\geq p+1$ is  similar by symmetry.

Assume now $l_i\geq2p+1$, for some $i\in [[1,k-1]]$. Let $j=i+\sum_{p=0}^{i-1}{l_p}$. We thus have $t(H)_j=1$, $t(H)_{j+l_i+1}=1$, $t(H)_{j+p+1}=0$,  $p$ $0$s between $t(H)_{j}$ and $t(H)_{j+p+1}$
and  $l_i-p-1$ $0$s between $t(H)_{j+p+1}$ and $t(H)_{j+l_i+1}$. Note that $l_i-p-1\geq p$. Therefore for any vertex $x$ of $H$ we have thus $x+\delta_{j+p+1}$ $\in V(\Gamma_n^p)$ and $H$ is not maximal, a contradiction.

Conversely, consider a vertex $z=0^{l_0}10^{l_1}\dots 10^{l_i}\dots10^{l_{k}}$, where $\sum_{i=0}^k{l_i}=n-k$, $l_0, l_k \in [[0,p]]$, and $l_i\in [[p,2p]]$ for $i \in [[1,k-1]]$. Then the couple $(t(H)=z,b(H)=0^n)$ defines a unique hypercube $H$ in $Q_n$ isomorphic 
to $Q_k$. Clearly all vertices of $H$ are Fibonacci $p$-strings. Note that for any $i\notin \supp(H)$  the string $z+\delta_i$ is not a Fibonacci $p$-string so $H$ is maximal.
\end{proof}
\begin{theorem}
\label{th:FiboMax}
If $0\leq k \leq n$, and $h_k(\Gamma_n^p)$ is the number of maximal hypercubes of dimension $k$ in $\Gamma_n^p$, then
$$h_k(\Gamma_n^p)=\binom{k+1}{n-(p+1)k+p}_{p}\,.$$
\end{theorem}
\begin{proof}
This is clearly true for $k=0$ so assume $k\geq1$. Since maximal hypercubes of $\Gamma_n^p$ are characterized by their top vertices, let us consider the set $T$ of strings which can be written as $0^{l_0}10^{l_1}\dots 10^{l_i}\dots10^{l_k}$ where $\sum_{i=0}^k{l_i}=n-k$, $l_0, l_k\in [[0,p]]$, and $l_i\in [[p,2p]]$ for $i \in [[1,k-1]]$. Let $l'_i =l_i-p$ for $i \in [[1,k-1]]$, $l'_0 =l_0$, and $l'_k =l_k$. We have thus a one to one mapping between $T$ and the set of strings $D=\{0^{l'_0}10^{l'_1}\dots 10^{l'_i}\dots10^{l'_{k}}\}$, where $\sum_{i=0}^k{l'_i}=n-k-(k-1)p=n-(p+1)k+p$ 
with $l'_i\in [[0,p]]$ for $i \in \{0,\dots,k\}$. 
This set is in bijection with  the set $E=\{1^{l'_0}01^{l'_1}\dots 01^{l'_i}\dots01^{l'_{k}}\}$. Since $E$ is the set of $(p+1)$-th Fibonacci strings of length $n-kp+p$ and Hamming weight $n-(p+1)k+p$  we are done by Theorem~\ref{th:weight}.
\end{proof}\qed

\begin{corollary} 
\label{co:FiboMax211}
Let
$$
H_{\Gamma_n^p}(x)=\sum_{k\geq 0}h_k(\Gamma_n^p)x^k
$$ 
denote the enumerator polynomial of the maximal hypercubes of dimension $k$ in $\Gamma_n^p$. Then
\begin{equation}\label{eqn:HGamma}
H_{\Gamma_n^p}(x) = x\sum_{i=1}^{p+1}H_{\Gamma_{n-p-i}^p}(x) \ \ \ \ \ \text{for } n\geq2p+1
\end{equation}

with $H_{\Gamma_0^p}(x) =0$, $H_{\Gamma_n^p}(x) =nx$ for $n\in[[1,p+1]]$ and $H_{\Gamma_n^p}(x) =\frac{(n-p-1)(n-p)}{2}x^2+(2p-n+2)x$ for $n\in [[p+2,2p]]$.

The  generating function of the sequence $(H_{\Gamma_n^p}(x))_{n\geq 0}$ is 
\begin{equation}\label{eqn:HGx}
\sum_{n \ge 0}H_{\Gamma_n^p}(x)t^n=\frac{1+xt(1+t+t^2+\dots+t^p)(1+t+t^2+\dots+t^{p-1})}{1-xt^{p+1}(1+t+t^2+\dots+t^p)} \, . 
\end{equation}
\end{corollary}
\begin{proof}
Let $k\geq1$ and $n\geq2p+1$. By Theorem~\ref{th:FiboMax} and generalized Pascal identity (\ref{eqn:pascal}) we obtain 
\begin{eqnarray*}\label{eqn:brut}
h_k(\Gamma_n^p)&=&\binom{k+1}{n-(p+1)k+p}_{p}=\sum_{i=0}^{p}{\binom{k}{n-(p+1)k+p-i}_{p}}\\
&=&\sum_{i=0}^{p}{\binom{k}{n-p-i-1-(p+1)(k-1)+p}_{p}}=\sum_{i=0}^{p}{h_{k-1}(\Gamma_{n-p-i-1}^p)}\,.
\end{eqnarray*} 
Note that $h_0(\Gamma_n^p)=0$ for $n\neq 0$.  
The recurrence relation for $H_{\Gamma_n^p}(x)$ follows.


For $n\in[[1,p+1]]$ the Fibonacci $p$-cube $\Gamma_n^p$ is the star $S_n$ thus $H_{\Gamma_n^p}(x)=nx$.
For $n\in [[p+2,2p]]$  by Theorem~\ref{th:FiboMax}  $h_1(\Gamma_n^p)=2p-n+2$, $h_2(\Gamma_n^p)=\frac{(n-p-1)(n-p)}{2}$ and $h_k(\Gamma_n^p)=0$ for $k\geq 3$.

Let $H(x,t)$ be the right hand side of (\ref{eqn:HGx}) and  $$A(x,t)=\frac{1+t+t^2+\dots+t^p}{1-xt^{p+1}(1+t+t^2+\dots+t^p)}\,.$$
 An immediate calculation gives the following equality
\begin{equation}\label{eq:A(x,t)}
t^p H(x,t)=  A(x,t)-(1+t+t^2+\dots+t^{p-1})\,.
 \end{equation}
For any integers $k$ and $n$ the coefficient of $x^kt^n$  in $ H(x,t)$ is that of $x^kt^{n+p}$  on the right hand side of (\ref{eq:A(x,t)}) thus in $A(x,t)$.

By expansion of a $\frac{1}{1-u}$
\begin{equation*}
A(x,t)=\sum_{m\geq0}{(1+t+t^2+\dots+t^p)^{m+1}x^mt^{(p+1)m}}=\sum_{m\geq0}\sum_{r=0}^{p(m+1)}{\binom{m+1}{r}_{p}}{x^mt^{(p+1)m+r}}\,.
\end{equation*}
A summand  contributes to $x^kt^{n+p}$ if and only if  $m=k$ and $(p+1)m+r=n+p$. Furthermore this contribution is $\binom{m+1}{r}_{p}=\binom{k+1}{n-(p+1)k+p}_{p}$.
 
Therefore, by Theorem~\ref{th:FiboMax} $,
H(x,t)=\sum_{n\geq0 }{\sum_{k\geq0}}{h_k(\Gamma_n^p)}x^kt^n
$
and by identification we obtain identity~(\ref{eqn:HGx}).

\end{proof}
\qed

It would be particularly interesting to complete these results by determining the maximal hypercubes of $\Gamma_n^{(p)}$.

\bibliographystyle{plain}  

\bibliography{fibo-bib3}

\begin{thebibliography}{10}

\bibitem{A-1990}
George Andrews.
\newblock Euler's ``{E}xemplum {M}emorabile {I}nductionis {F}allacis'' and
  q-{T}rinomial coefficients.
\newblock {\em Journal of The American Mathematical Society}, 3:653--653, 07
  1990.

\bibitem{BO-2020}
Hac\`ene Belbachir and Ryma Ould-Mohamed.
\newblock Enumerative properties and cube polynomials of {T}ribonacci cubes.
\newblock {\em Discrete Mathematics}, 343:111922, 08 2020.

\bibitem{BKS-2003}
Bo{\v{s}}tjan Bre{\v{s}}ar, Sandi Klav{\v{z}}ar, and Riste {\v{S}}krekovski.
\newblock The cube polynomial and its derivatives: the case of median graphs.
\newblock {\em Electron. J. Combin.}, 10:Research Paper 3, 11 pp. (electronic),
  2003.

\bibitem{A-1974}
Louis Comtet.
\newblock {\em Advanced Combinatorics}.
\newblock Reidel, 1974.

\bibitem{EA-1997}
Karen Egiazarian and Jaakko Astola.
\newblock On {G}eneralized {F}ibonacci {C}ubes and {U}nitary {T}ransforms.
\newblock {\em Applicable Algebra in Engineering, Communication and Computing},
  8:371--377, 1997.

\bibitem{EKM-2023}
\"{O}mer E\u{g}ecio\u{g}lu, Sandi Klav\v{z}ar, and Michel Mollard.
\newblock {\em Fibonacci {C}ubes with {A}pplications and {V}ariations}.
\newblock WORLD SCIENTIFIC, 2023.

\bibitem{ESS-2021e}
\"{O}mer E\u{g}ecio\u{g}lu, Elif {Sayg\i}, and Z\"{u}lf\"{u}kar {Sayg\i}.
\newblock Alternate {L}ucas cubes.
\newblock {\em Internat. J. Found. Comput. Sci.}, 32(7):871--899, 2021.

\bibitem{E-1801}
Leonhard Euler.
\newblock De evolutione potestatis polynomialis cuiuscunque
  {$(1+x+x^2+x^3+x^4+etc.)^n$}.
\newblock {\em Nova Acta Academiae Scientiarum Imperialis Petropolitanae},
  pages 47--57, 1801.

\bibitem{F-2015}
Nour-Eddine Fahssi.
\newblock Some identities involving polynomial coefficients.
\newblock {\em Fibonacci Quarterly}, 54, 07 2015.

\bibitem{HIK-2011}
Richard Hammack, Wilfried Imrich, and Sandi Klav\v{z}ar.
\newblock {\em Handbook of {P}roduct {G}raphs}.
\newblock Discrete Mathematics and its Applications. CRC Press, Boca Raton, FL,
  second edition, 2011.

\bibitem{HC-1993}
W.~J. Hsu and M.~J. Chung.
\newblock Generalized {F}ibonacci {C}ubes.
\newblock In {\em 1993 International Conference on Parallel Processing -
  ICPP'93}, volume~1, pages 299--302, 1993.

\bibitem{H-1993a}
Wen-Jing Hsu.
\newblock Fibonacci cubes: a new interconnection topology.
\newblock {\em IEEE Trans. Parallel Distrib. Syst.}, 4:3--12, 1993.

\bibitem{IKR-2012a}
Aleksandar Ili\'{c}, Sandi Klav\v{z}ar, and Yoomi Rho.
\newblock Generalized {F}ibonacci cubes.
\newblock {\em Discrete Math.}, 312(1):2--11, 2012.

\bibitem{H-1969}
V.~E.~Hoggatt Jr. and Marjorie Bicknell.
\newblock Diagonal {S}ums of {G}eneralized {P}ascal {T}riangles.
\newblock {\em The Fibonacci Quarterly}, 7(4):341--358, 1969.

\bibitem{K-2005}
Sandi Klav\v{z}ar.
\newblock On median nature and enumerative properties of {F}ibonacci-like
  cubes.
\newblock {\em Discrete Math.}, 299(1):145--153, 2005.

\bibitem{K-2013a}
Sandi Klav\v{z}ar.
\newblock Structure of {F}ibonacci cubes: a survey.
\newblock {\em J. Comb. Optim.}, 25(4):505--522, 2013.

\bibitem{KM-2012a}
Sandi Klav\v{z}ar and Michel Mollard.
\newblock Cube polynomial of {F}ibonacci and {L}ucas cubes.
\newblock {\em Acta Appl. Math.}, 117:93--105, 2012.

\bibitem{KM-2019a}
Sandi Klav\v{z}ar and Michel Mollard.
\newblock Daisy cubes and distance cube polynomial.
\newblock {\em European J. Combin.}, 80:214--223, 2019.

\bibitem{LWC-1994}
Jenshiuh Liu, Wen~Jing Hsu, and Moon Chung.
\newblock Generalized {F}ibonacci cubes are mostly {H}amiltonian.
\newblock {\em Journal of Graph Theory}, 18:817 -- 829, 12 1994.

\bibitem{D-1756}
Abraham~de Moivre.
\newblock {\em The Doctrine of Chances}.
\newblock London, 1756.

\bibitem{M-2012a}
Michel Mollard.
\newblock Maximal hypercubes in {F}ibonacci and {L}ucas cubes.
\newblock {\em Discrete Appl. Math.}, 160(16-17):2479--2483, 2012.

\bibitem{M-2025a}
Michel Mollard.
\newblock Distance cube polynomials of {F}ibonacci and {L}ucas-run graphs.
\newblock {\em Discrete Appl. Math.}, 376:178--185, 2025.

\bibitem{M-2025b}
Michel Mollard.
\newblock Some new results about {F}ibonacci p-cubes (preprint).
\newblock 2025.
\newblock arXiv 2502.07520 [math.CO].

\bibitem{M-2019}
Emanuele Munarini.
\newblock Pell graphs.
\newblock {\em Discrete Math.}, 342(8):2415--2428, 2019.

\bibitem{OZ-2013}
Lifeng Ou and Heping Zhang.
\newblock Fibonacci (p, r)-cubes which are median graphs.
\newblock {\em Discrete Applied Mathematics}, 161(3):441--444, 2013.

\bibitem{RV-2016}
Yoomi Rho and Aleksander Vesel.
\newblock Linear recognition of generalized {F}ibonacci cubes {$Q_h(111)$}.
\newblock {\em Discrete Math. Theor. Comput. Sci.}, 17(3):349--362, 2016.

\bibitem{SE-2017a}
Elif {Sayg\i} and \"{O}mer E\u{g}ecio\u{g}lu.
\newblock {$q$}-{C}ube enumerator polynomial of {F}ibonacci cubes.
\newblock {\em Discrete Appl. Math.}, 226:127--137, 2017.

\bibitem{SE-2018a}
Elif {Sayg\i} and \"{O}mer E\u{g}ecio\u{g}lu.
\newblock {$q$}-{C}ounting hypercubes in {L}ucas cubes.
\newblock {\em Turkish J. Math.}, 42(1):190--203, 2018.

\bibitem{WG-2003}
H.C. Wasserman and S.A. Ghozati.
\newblock Generalized linear recursive networks: topological and routing
  properties.
\newblock {\em Computers \& Electrical Engineering}, 29(1):121--134, 2003.

\bibitem{WY-2022}
Jianxin Wei and Yujun Yang.
\newblock Fibonacci and {L}ucas {$p$}-cubes.
\newblock {\em Discrete Appl. Math.}, 322:365--383, 2022.

\bibitem{S-1996}
Norma Zagaglia~Salvi.
\newblock On the existence of cycles of every even length on generalized
  {F}ibonacci cubes.
\newblock {\em Matematiche (Catania)}, 51(suppl.):241--251 (1997), 1996.

\end{thebibliography}

\end{document}